\documentclass[11pt,reqno]{amsart}
\usepackage{amsfonts,amssymb,amscd,amsmath,enumerate,verbatim,calc}
\usepackage{amsmath}
\usepackage{amssymb}
\usepackage{amscd}
\usepackage{color}
\usepackage{mathtools}
\usepackage[colorlinks,pagebackref=true]{hyperref}
\makeatletter

\makeatletter
\def\widebreve#1{\mathop{\vbox{\m@th\ialign{##\crcr\noalign{\kern3\p@}%
				\brevefill\crcr\noalign{\kern3\p@\nointerlineskip}%
				$\hfil\displaystyle{#1}\hfil$\crcr}}}\limits}

\def\brevefill{$\m@th \setbox\z@\hbox{$\braceld$}%
	\bracelu\leaders\vrule \@height\ht\z@ \@depth\z@\hfill\braceru$}
\makeatletter

\def\@citecolor{blue}
\def\@linkcolor{blue}
\def\@urlcolor{blue}
\def\@urlcolor{blue}

\def\NZQ{\mathbb}               
\def\NN{{\NZQ N}}

\def\ZZ{{\NZQ Z}}
\def\RR{{\NZQ R}}
\def\CC{{\NZQ C}}

\def\mfp{\mathfrak p}

\def \mm{\mathfrak {m}}

\def\Min{\operatorname{Min}}
\def\sat{\operatorname{sat}}

\def\grade{\operatorname{grade}}

\newtheorem{Theorem}{Theorem}[section]
\newtheorem{Lemma}[Theorem]{Lemma}

\newtheorem{Proposition}[Theorem]{Proposition}
\newtheorem{Remark}[Theorem]{Remark}

\newtheorem{Example}[Theorem]{Example}

\newtheorem{Definition}[Theorem]{Definition}

\let\epsilon\varepsilon
\let\phi=\varphi
\let\kappa=\varkappa

\begin{document}
	\title{Multiplicities of weakly graded families of ideals}
	\author{Parangama Sarkar}
	\thanks{The author was partially supported by SERB POWER Grant with Grant No. SPG/2021/002423.}
	\address{Parangama Sarkar,
		Department of Mathematics,  
		Indian Institute of Technology Palakkad,
		Nila Campus, Kanjikode, 
		Palakkad-678623, Kerala, India.}
	\email{parangamasarkar@gmail.com, parangama@iitpkd.ac.in}
	\begin{abstract} In this article, we extend the notion of multiplicity for weakly graded families of ideals which are bounded below linearly. In particular, we show that the limit $\displaystyle e_W(\mathcal I):=\lim\limits_{n\to\infty}d!\frac{\ell_R(R/I_n)}{n^d}$ exists where $\mathcal I=\{I_n\}$ is a bounded below linearly weakly graded family of ideals in a Noetherian local ring $(R,\mm)$ of dimension $d\geq 1$ with $\dim(N(\hat{R}))<d$. Furthermore, we prove that ``volume=multiplicity" formula and Minkowski inequality hold for such families of ideals. 
	We explore some properties of $e_W(\mathcal J)$
		for weakly graded families of ideals of the form $\mathcal J=\{(I_n:K)\}$ where $\{I_n\}$ is an $\mm$-primary graded family of ideals.
		We provide a necessary and sufficient condition for the equality in Minkowski inequality for the weakly graded families of ideals of the form $\{(I_n:K)\}$ where $\{I_n\}$ is a bounded filtration. Moreover, we generalize a result of Rees characterizing the inclusion of ideals with the same multiplicities for the above families of ideals.  Finally, we investigate the asymptotic behaviour of the length function  $\ell_R(H_{\mm}^0(R/(I_n:K)))$ where $\{I_n\}$ is a filtration of ideals (not necessarily $\mm$-primary).
	\end{abstract}
	
	
	\maketitle
	
	\section{introduction}
	Let $(R,\mm)$ be a Noetherian local ring of dimension $d$ and $I$ be an ideal in $R.$ If $I$ is $\mm$-primary then extending the work of Hilbert \cite{hilb}, Samuel \cite{samuel} proved that for all large $n,$ the  Hilbert-Samuel function of $I,$ $\ell_R({R}/{I^n})$ coincides with a polynomial in $n$ of degree $d$ (here the length of an $R$-module $M$ is denoted by $\ell_R(M)$)  and $e(I):=\displaystyle\lim\limits_{n\to\infty}d!\frac{\ell_R({R}/{I^n})}{n^d}$ is a positive integer. The positive integer $e(I)$ is called the multiplicity of $I.$ 
	
	Some easy examples show that the limit $\displaystyle e(\mathcal I):=\lim\limits_{n\to\infty}d!\frac{\ell_R({R}/{I_n})}{n^d}$ can be an irrational number for a non-Noetherian filtration $\mathcal I=\{I_n\}$ of $\mm$-primary ideals \cite{CSS}. There are examples of graded families of $\mm$-primary ideals for which the above limit does not exist in a Noetherian local ring \cite{C2014}. The question of whether such limits exist has been considered by several mathematicians (see Ein, Lazarsfeld and Smith \cite{ELS}, Musta\c{t}\u{a} \cite{Mus}). If $(R,\mm)$ is a local domain which is essentially of finite type over an algebraically closed field $k$ with $R/\mm=k$ then Lazarsfeld and Musta\c{t}\u{a} \cite{LM} proved that the above limit exists for any graded family of $\mm$-primary ideals using a method introduced by Okounkov \cite{Ok}. In \cite{C2014}, Cutkosky proved that in a Noetherian local ring  $(R,\mm)$ of dimension $d\geq 1$, $e(\mathcal I)$ exists for any graded family $\mathcal I$ of $\mm$-primary ideals  if and only if $\dim N(\hat{R})<\dim R$ where $\hat{R}$ is the $\mm$-adic completion of $R$.  He also showed that the ``volume=multiplicity" formula holds for graded families of $\mm$-primary ideals \cite{C2014}.
	This leads naturally to the question of whether the limit $\displaystyle e_W(\mathcal I):=\lim\limits_{n\to\infty}d!\frac{\ell_R({R}/{I_n})}{n^d}$  exists for a bounded below linearly weakly graded family of ideals $\mathcal I=\{I_n\}$ (see Definition \ref{def}).
	The main objective of this article is to prove the existence of the above limit and  investigate some of the classical results about multiplicities for a bounded below linearly weakly graded family of ideals, namely Minkowski inequality and Rees' characterization of equality of multiplicities. 
		
		Minkowski inequality of multiplicities of ideals was  originated in Teissier's work on equisingularity. In his Carg\`{e}se paper \cite{T1}, he conjectured that the Minkowski inequality $e(IJ)^{1/d}\leq e(I)^{1/d}+e(J)^{1/d}$ holds for any two $\mm$-primary ideals $I$ and $J$ in $R$ with $d\geq 1$ and proved it when $R$ is reduced Cohen-Macaulay and contains $\mathbb Q$ \cite{T2}. Later Rees and Sharp  proved it in full generality \cite{RS}.  In \cite{CSS}, Cutkosky, Srinivasan and the author proved that the Minkowski inequality holds for filtrations (not necessarily Noetherian) of $\mm$-primary ideals. Minkowski equality  is further explored in  \cite{C2021} and \cite{CS}.
	
Rees used multiplicities of $\mm$-primary ideals to investigate the numerical characterization of reductions. He showed that if $R$ is a formally equidimensional local ring and $J\subset I$ are $\mm$-primary ideals in $R$ then the integral closures of the Rees algebras $R[It]$ and $R[Jt]$ in the polynomial ring $R[t]$ are same if and only if $e(I)=e(J)$ \cite{Rees4}.  In \cite{CSS}, Cutkosky, Srinivasan and the author considered the equality of $e(\mathcal I)$ and $e(\mathcal J)$ for filtrations $\mathcal I$ and $\mathcal J$ (not necessarily Noetherian) of $\mm$-primary ideals and generalized the result of Rees \cite{Rees4}.  Therefore it is  natural to analyse the Minkowski inequality and explore the equality of multiplicities of bounded below linearly weakly graded families of ideals. In this direction, in section $3$, we show the following.
	\begin{Theorem}\label{First}
		Let $\mathcal I=\{I_n\}$ be a bounded below linearly weakly graded family of ideals in a Noetherian local ring $(R,\mm)$ of dimension $d\geq 1$ with $\dim(N(\hat{R}))<d$. Then the following hold.
		\begin{enumerate} 
			\item[$(i)$] The limit  $e_W(\mathcal I):=\lim\limits_{n\to\infty}d!\ell_R(R/I_n)/n^d$ exists.
			\item[$(ii)$] (Volume=Multiplicity) $e_W(\mathcal I)=\lim\limits_{n\to\infty}e(I_n)/n^d.$
			\item[$(iii)$] (Minkowski inequality)	Let $\mathcal J=\{J_n\}$ be a bounded below linearly weakly graded family of ideals in $R$ and $\mathcal{IJ}=\{I_nJ_n\}$. Then $\mathcal{IJ}$ is a bounded below linearly weakly graded family of ideals in $R$ and
			$$e_W(\mathcal{IJ})^{1/d}\leq e_W(\mathcal{I})^{1/d}+e_W(\mathcal{J})^{1/d}.$$
			\item[$(iv)$]	Let $\mathcal J=\{J_n\}$ be a bounded below linearly weakly graded family of ideals in $R$ such that $I_n\subset J_n$ and  $\overline{I_n}=\overline{J_n}$ for all $n\gg 0$ then $e_W(\mathcal{I})=e_W(\mathcal{J})$.
		\end{enumerate}
	\end{Theorem}
We define the multiplicity of a  bounded below linearly weakly graded family of ideals $\mathcal I=\{I_n\}$ to be the limit $e_W(\mathcal I)=\displaystyle\lim\limits_{n\to\infty}d!\ell_R(R/I_n)/n^d$.  The inequality in part $(iii)$ can be strict and the converse of part $(iv)$ is not true in general (Examples \ref{countmin} and \ref{e2}). Recently, in \cite{DM}, the authors considered the existence of the limit $\lim\limits_{n\to\infty}\ell_R(R/I_n)/n^d$, the ``volume=multiplicity" formula and Minkowski inequality for weakly graded families of ideals along with some different types of families of ideals. Theorem \ref{First}, $(i)-(iii)$ provide alternative proofs of Theorems $10.11$,  $10.14$ and $10.16$ in \cite{DM}. 
	
	For any $\mm$-primary ideal $I$, the length $\ell_R((I:\mm)/I)$ gives the  number of irreducible components of an irredundant irreducible decomposition of $I$ and this number is independent of the choice of the decomposition \cite{No}. Thus for any graded family $\{I_n\}$ of $\mm$-primary ideals, the asymptotic behavior of the length function $\ell_R(R/(I_n:\mm))$  is of significant interest. In section $4$, we consider the weakly graded family of ideals of the form $\mathcal J=\{(I_n:K)\}$ where $\{I_n\}$ is a graded family of $\mm$-primary ideals in $R$.  We show that the multiplicity $e_W(\mathcal J)$ is always bounded above by the multiplicity $e(\mathcal I)$ of $\mathcal I$ and in some cases, the limit achieves the upper bound (Proposition \ref{tech}). We generalize a result due to Rees \cite{Rees4} for the weakly graded family of ideals  $\{(I_n:K)\}$ and explore the Minkowski equality for such families of ideals.
	\begin{Theorem}
		Let $(R,\mm)$ be an analytically irreducible local domain and $K$ be an ideal in $R$.
		\begin{enumerate}
			\item[$(i)$] Let $\mathcal I=\{I_n\}$ be a real bounded filtration of $\mm$-primary ideals and $\mathcal J=\{J_n\}$ be a filtration of $\mm$-primary ideals in $R$ such that $I_n\subset J_n$ for all $n$. Then  $\overline{R[\mathcal I]}=\overline{R[\mathcal J]}$ if and only if $$\lim\limits_{n\to\infty}\ell_R(R/(I_n:K))/n^d=\lim\limits_{n\to\infty}\ell_R(R/(J_n:K))/n^d.$$
			\item[$(ii)$] (Minkowski equality) Let $\mathcal I=\{I_n\}$ and $\mathcal J=\{J_n\}$  be two integral bounded filtrations of $\mm$-primary ideals in $R$. Then equality holds in $(iii)$ of Theorem \ref{First} for the bounded below linearly  weakly graded families of ideals $\{(I_n:K)\}$ and $\{(J_n:K)\}$ if and only if there exist positive integers $a, b$ such that $\overline{\sum\limits_{n\geq 0} I_{an}t^n}=\overline{\sum\limits_{n\geq 0} J_{bn}t^n}$ where the integral closures are in $R[t]$.
		\end{enumerate}
	\end{Theorem}	
	We conclude this section by showing that for a weakly graded family (not necessarily bounded below linearly) of ideals of the form $\{(I_{n}:K)\}$, the limit  $\lim\limits_{n\to\infty}\ell_R(H_\mm^0(R/(I_{n}:K)))/{n}^{d}$ exists and bounded above by the epsilon multiplicity of the filtration $\{I_n\}$ under some extra assumptions on $\{I_n\}$ and $K$. 
\begin{Theorem}
	Let $(R,\mm)$ be an analytically unramified local ring of dimension $d\geq 1$ and $K$ be an ideal in $R$. 
	\begin{enumerate} 
		\item[$(i)$] Let 
		$\mathcal I=\{I_n\}$ be a filtration of ideals in $R$ which satisfies $A(r)$ condition for some $r\in\ZZ_{>0}$. Then the limit $$\lim\limits_{n\to\infty}\ell_R(H_\mm^0(R/(I_{n}:K)))/{n}^{d}$$ exists. 
		\\Suppose $K$ is an $\mm$-primary ideal in $R$. Then $(I_{n}:K)^{\sat}=I_n^{\sat}$ for all $n\geq 1$ and $$\lim\limits_{n\to\infty}d!\ell_R(H_\mm^0(R/(I_{n}:K)))/{n}^{d}\leq \epsilon(\mathcal I).$$ 
		\item[$(ii)$] Let $\mathcal I=\{I_n\}$ be a Noetherian filtration of ideals in $R$ with  $\grade(I_1)\geq 1$ and $K$ be an $\mm$-primary ideal in $R$. Then $$\lim\limits_{n\to\infty}d!\ell_R(H_\mm^0(R/({I_{n}}:K)))/{n}^{d}= \epsilon(\mathcal I).$$ 
		In particular, if $I$ is an ideal in $R$ with $\grade(I)\geq 1$ and $K$ is an $\mm$-primary ideal in $R$ then $$\lim\limits_{n\to\infty}d!\ell_R(H_\mm^0(R/(\overline{I^{n}}:K)))/{n}^{d}=\lim\limits_{n\to\infty}d!\ell_R(H_\mm^0(R/({I^{n}}:K)))/{n}^{d}.$$
	\end{enumerate}
\end{Theorem}	
	\section{notation and definitions}
	We denote the nonnegative integers by $\NN$,  the positive integers by $\ZZ_{>0}$ and the set of   the positive real numbers by $\RR_{>0}$. 
	For a real number $x$, the smallest integer that is greater than or equal to $x$  is denoted by $\lceil x\rceil$. 
	
	Let $(R,\mm)$ be a Noetherian local ring of dimension $d\geq 1$. We denote the set $R\setminus \bigcup\limits_{P\in\Min R}P$ by $R^o$ and the $\mm$-adic completion of $R$ by $\hat{R}$.
	\begin{Remark}\label{nzd}{\rm
		Since ${\hat{R}}$  is a flat $R$-algebra, by \cite[Theorem 9.5]{Mat}, the going-down theorem holds and hence contraction of any minimal prime of $\hat{R}$ is a minimal prime of $R$. Thus for any $c\in R^o$, we have $c\in {\hat{R}}^o$.}
	\end{Remark}	
	\begin{Definition} A graded family $\mathcal I=\{I_n\}_{n\in\NN}$ of ideals in a ring $R$ is a collection of ideals in $R$ such that $I_0=R$ and $I_mI_n\subset I_{m+n}$ for all $m,n\in \NN$. 
	\\A graded family of ideals in $R$ is called a filtration if $I_n\subset I_m$ for all $m,n\in \NN$ with $m\leq n$.
	\end{Definition}
	\begin{Definition} {\rm\cite{DM}} A family of ideals $\mathcal I=\{I_n\}_{n\in\NN}$  in $R$ is called a weakly graded family of ideals  if $I_0=R$ and there exists an element $c\in R^o$ such that $cI_mI_n\subset I_{m+n}$ for all $m,n\geq 1$. 
\end{Definition}
\begin{Definition}{\rm\cite{DM}} \label{def} A weakly graded family of ideals $\mathcal I=\{I_n\}_{n\in\NN}$ in $R$  is called a bounded below linearly weakly graded family of ideals if there exists an integer $s\in \mathbb Z_{>0}$ such that $\mm^{sn}\subset I_n$ for all $n\geq 1$.
\end{Definition}
	\begin{Remark}{\em
			Let $(R,\mm)$  be a Noetherian local ring  of dimension $d\geq 1$ and $\mathcal J=\{J_n\}$ be a (bounded below linearly) weakly graded family of ideals. Let $c\in R^o$ such that $cJ_mJ_n\subset J_{m+n}$.
			\begin{enumerate}
				\item Suppose $\{J_n\}$ is a (bounded below linearly) graded family  of ideals in $R$ such that $J_1\cap R^o\neq\emptyset$. Let $a\in\RR_{>0}$ and $c\in J_1\cap R^o$. Then for all $m,n\geq 1$, we have $$cJ_{\lfloor{an}\rfloor}J_{\lfloor{am}\rfloor}\subset J_1J_{\lfloor{an}\rfloor+\lfloor{am}\rfloor}\subset J_{\lfloor{an}\rfloor+\lfloor{am}\rfloor+1}\subset J_{\lfloor a(n+m)\rfloor}.$$
				(If $\mm^{sn}\subset J_n$  then $\mm^{s(\lfloor{a}\rfloor+1) n}\subset J_{\lfloor{an}\rfloor}$ for all $n\geq 1$.) Hence $\{I_n=J_{\lfloor{an}\rfloor}\}$ is a (bounded below linearly) weakly graded family  of ideals in $R$. 
				\item Let  $A=\{K(1),\ldots,K(r)\}$ be a collection of ideals in $R$ such that $K(i)\cap R^o\neq\emptyset$ for all $1\leq i\leq r$. Then $\mathcal I=\{I_n=(J_{n}:K_n)\}$ is a (bounded below linearly) weakly graded family of ideals where $K_n$ varies in $A$. Let $c_i\in K(i) \cap R^o$  for all $1\leq i\leq r$ and $d=c_1\cdots c_r$. Then $cd^2 I_nI_m\subset I_{n+m}$ for all $m,n\geq 1$.
				\item  Let  $K$ be an ideal in $R$ such that $K\cap R^o\neq\emptyset$.  Then $\mathcal I=\{I_n=(J_n:K^{n+1})\}$  is a (bounded below linearly) weakly graded family of ideals.  Let $d\in K\cap R^o$. Then $cdI_nI_m\subset I_{n+m}$ for all $m,n\geq 1$.
				\item Let $S=\{c \in R^o : cJ_mJ_n\subset J_{m+n} \mbox{ for all } m,n\geq 1 \}$. The graded family $\{cJ_n\}$ is not  necessarily a Noetherian graded family for any $c\in S$.
			\end{enumerate}
		}
	\end{Remark}	
	
	We denote the integral closure of an ideal $I$ in $R$ by $\overline I$. Let $\mathcal I=\{I_n\}$ be a graded family of ideals in $R$. We say $\mathcal I=\{I_n\}$ is a Noetherian graded family if the graded $R$-algebra
	$
	R[\mathcal I]=\bigoplus_{n\in\NN}I_nt^n
	$ is a finitely generated $R$-algebra. Otherwise, we say  $\mathcal I=\{I_n\}$  is non-Noetherian.
	Let $\overline {R[\mathcal I]}$ denote the integral closure of $R[\mathcal I]$ in the polynomial ring $R[t]$. It is shown in \cite[Lemma 3.6]{CS} that for a filtration $\mathcal I=\{I_n\}$,  the integral closure of $R[\mathcal I]$ in $R[t]$ is
	$
	\overline{R[\mathcal I]}=\bigoplus_{m\ge 0}J_mt^m
	$
	where $\{J_m\}$ is the filtration
	$$J_m=\{f\in R\mid f^r\in \overline{I_{rm}}\mbox{ for some }r>0\}.$$ 
	Let $(R,\mm)$ be a Noetherian local domain of dimension $d$ with quotient field $K$. Let $\nu$ be a discrete valuation of $K$ with valuation ring $\mathcal O_{\nu}$ and maximal ideal $m_{\nu}$. Suppose that $R\subset \mathcal O_{\nu}$. Then for all $n\in \NN$, the valuation ideals are defined as
	$$
	I(\nu)_n=\{f\in R\mid \nu(f)\ge n\}=m_{\nu}^n\cap R.
	$$
	\begin{Definition} A discrete valued filtration of $R$ is a filtration $\mathcal I=\{I_n\}$ such that there exist discrete valuations $\nu_1,\ldots,\nu_r$ of $K$ and $a_1,\ldots, a_r\in \RR_{>0}$ such that $R\subset \mathcal O_{{\nu}_i}$ for all $1\leq i\leq r$ and for all $n\in \NN$,
	$$I_n=I(\nu_1)_{\lceil na_1\rceil}\cap\cdots\cap I(\nu_r)_{\lceil na_r\rceil}.$$ \end{Definition} 
	A divisorial valuation of $R$ (\cite[Definition 9.3.1]{SH}) is a valuation $\nu$ of $K$ such that if $\mathcal O_{\nu}$ is the valuation ring of $\nu$ with maximal ideal $\mathfrak m_{\nu}$, then $R\subset O_{\nu}$ and if $\mfp=\mathfrak m_{\nu}\cap R$ then $\mbox{trdeg}_{\kappa(\mfp)}\kappa(\nu)={\rm ht}(\mfp)-1$, where $\kappa(\mfp)$ is the residue field of $R_{\mfp}$ and $\kappa(\nu)$ is the residue field of $O_{\nu}$. Every divisorial valuation $\nu$ is a  discrete valuation  \cite[Theorem 9.3.2]{SH}. 
	\begin{Definition} A divisorial filtration of $R$ is a discrete valued filtration 
	$$
	\{I_n=I(\nu_1)_{\lceil na_1\rceil}\cap\cdots\cap I(\nu_r)_{\lceil na_r\rceil}\}
	$$
	where all the discrete valuations $\nu_1,\ldots,\nu_r$ are divisorial valuations.  A divisorial filtration is called integral if $a_i\in \ZZ_{>0}$ for all $1\leq i\leq r$.
\end{Definition} 
	\begin{Definition} A filtration of ideals $\mathcal J$ in $R$ is called a bounded filtration if there exists a divisorial filtration $\mathcal C$ such that $\overline{R[\mathcal J]}=R[\mathcal C]$. 
	 A bounded filtration  $\mathcal J$ is called an integral bounded filtration if $\overline{R[\mathcal J]}=R[\mathcal C]$ for some integral divisorial filtration $\mathcal C$.
	\end{Definition} 
	Following the same lines of the proof of \cite[Lemma 5.7]{C2021}, we get 
	\begin{Lemma}\label{integrallyclosed}
		If $\mathcal I$ is a discrete valued filtration in a Noetherian local domain $R$ then $\overline{R[\mathcal I]}=R[\mathcal I]$.
	\end{Lemma}
	Let $(R,\mm)$ be a Noetherian local ring. For an ideal $I$ in $R$, the saturation of $I$, denoted by $I^{\sat}$, is defined as $I^{\sat}=I:\mm^{\infty}=\bigcup\limits_{n\geq 1}(I:\mm^n)$. 
	\begin{Definition}{\rm($A(r)$ Condition)} A graded family of ideals $\{I_n\}$ in $R$ is said to satisfy $A(r)$ for some $r\in\ZZ_{>0}$ if for all $n\geq 1$, $I_n^{\sat}\cap \mm^{rn}=I_n\cap \mm^{rn}$.
		\end{Definition} Any discrete valued filtration in a Noetherian local domain satisfies $A(r)$ for some $r\in\ZZ_{>0}$ (\cite[Theorem 3.1]{CS2024}).  
	
	The epsilon multiplicity of an ideal $I$ in a Noetherian local ring $(R,\mm)$  is defined in \cite{KV} to be $$
	\epsilon(I)=d!\limsup_n\frac{\ell_R(H^0_{\mm}(R/I^n))}{n^d}.$$ 
	In \cite[Corollary 6.3]{C2014}, it is shown that if $R$ is analytically unramified then the epsilon multiplicity of $I$ is a limit. Epsilon multiplicity of a filtration is introduced in \cite{CS2024}. Epsilon multiplicity of a filtration $\mathcal I=\{I_n\}$ satisfying $A(r)$ for some $r\in\ZZ_{>0}$ is a limit, i.e., $
\displaystyle\epsilon(\mathcal I)=\lim\limits_{n\to\infty}d!\frac{\ell_R(H^0_{\mm}(R/I_n))}{n^d}$ exists. For more details about $A(r)$ condition, see \cite{CS2024}. 
	\section{multiplicity of weakly graded family of ideals}
	In this section, we show the existence of the limit $e_W(\mathcal I)=\lim\limits_{n\to\infty}d!\ell_R(R/I_n)/n^d$ for a bounded below linearly weakly graded family of ideals $\mathcal I=\{I_n\}$  in a Noetherian local ring $(R,\mm)$ of dimension $d\geq 1$ with $\dim(N(\hat{R}))<d$. We generalize the ``volume=multiplicity" formula for  a bounded below linearly weakly graded family of ideals. We also prove Minkowski inequality and 
	show that this inequality can be strict in general.
	 We provide a sufficient condition for the equality of the multiplicities of two bounded below linearly weakly graded families of ideals. The following lemma is well-known. For the sake of completeness, we include the proof here.
	\begin{Lemma}\label{reduction}
			Let $\mathcal I=\{I_n\}$ be a family of ideals in a reduced Noetherian local ring $(R,\mm)$ of dimension $d\geq 1$ and there exists a positive integer $s$ such that $\mm^{sn}\subset I_n$ for all $n\geq 1$. Let $\Min(R)=\{P_1,\ldots,P_r\}$ and $R_i=R/P_i$ for all $1\leq i\leq r$. Then the following hold.
			\begin{enumerate}
				\item[$(i)$] The existence of the limits $\lim\limits_{n\to\infty}\ell_{R_i}(R_i/I_nR_i)/n^d$ for all $1\leq i\leq r$ imply the existence of the limit $\lim\limits_{n\to\infty}\ell_R(R/I_n)/n^d$.
				\item[$(ii)$] Suppose for all $1\leq i\leq r$, the limits  $\lim\limits_{n\to\infty}\ell_{R_i}(R_i/I_nR_i)/n^d$ exist and $$\displaystyle\lim\limits_{n\to\infty}\ell_{R_i}(R_i/I_nR_i)/n^d=\lim\limits_{n\to\infty}\frac{\lim\limits_{m\to\infty}\ell_{R_i}(R_i/I_n^mR_i)/m^d}{n^d}.$$ Then $$\displaystyle\lim\limits_{n\to\infty}d!\ell_R(R/I_n)/n^d=\displaystyle\lim\limits_{n\to\infty}e(I_n)/n^d.$$ 
		\end{enumerate}	
	\end{Lemma}	
	\begin{proof}
	$(i)$	Using the same lines of the proof of  \cite[Lemma 5.1]{C2013}, for all $n\geq 1$, we have \begin{equation}\label{known}|\sum\limits_{i=1}^r \ell_{R_i}(R_i/I_nR_i)-\ell_R(R/{I_n})|\leq Cn^{d-1}\end{equation} for some constant $C\in\ZZ_{>0}$. Hence  the existence of the limits 
	$\lim\limits_{n\to\infty}\ell_{R_i}(R_i/I_nR_i)/n^d$  for all $1\leq i\leq r$ imply the existence of the limit $\lim\limits_{n\to\infty}\ell_R(R/{I_n})/n^d$. 
	\\$(ii)$ Using part $(i)$ and equation (\ref{known}), we have
	 \begin{eqnarray*}
		\lim\limits_{n\to\infty}d!\ell_R(R/I_n)/n^d&=&\sum\limits_{i=1}^r \lim\limits_{n\to\infty}d!\ell_{R_i}(R_i/I_nR_i)/n^d=
		\sum\limits_{i=1}^r\lim\limits_{n\to\infty}\frac{\lim\limits_{m\to\infty}d!\ell_{R_i}(R_i/I_n^mR_i)/m^d}{n^d}
		 \\&=&
		\lim\limits_{n\to\infty}\frac{\sum\limits_{i=1}^r\lim\limits_{m\to\infty} d!\ell_{R_i}(R_i/I_n^mR_i)/m^d}{n^d}\\&=&	\lim\limits_{n\to\infty}\frac{\lim\limits_{m\to\infty} d!\ell_{R}(R/I_n^m)/m^d}{n^d}=\lim\limits_{n\to\infty}e(I_n)/n^d
		\end{eqnarray*} where the second last equality holds from equation (\ref{known}) for the family of ideals $\{I_n^m\}_{m\in\NN}$.
	\end{proof}	
	\begin{Theorem}\label{existence}
		Let $\mathcal I=\{I_n\}$ be a bounded below linearly weakly graded family of ideals in a Noetherian local ring $(R,\mm)$ of dimension $d\geq 1$ with $\dim(N(\hat{R}))<d$. Then the following hold.
		\begin{enumerate} 
			\item[$(i)$] The limit  $e_W(\mathcal I):=\lim\limits_{n\to\infty}d!\ell_R(R/I_n)/n^d$ exists.
			\item[$(ii)$] (Volume=Multiplicity) $e_W(\mathcal I)=\lim\limits_{n\to\infty}e(I_n)/n^d.$
			\item[$(iii)$] (Minkowski inequality)	Let $\mathcal J=\{J_n\}$ be a bounded below linearly weakly graded family of ideals in $R$ and $\mathcal{IJ}=\{I_nJ_n\}$. Then  $\mathcal{IJ}$ is a bounded below linearly weakly graded family of ideals in $R$ and
			$$e_W(\mathcal{IJ})^{1/d}\leq e_W(\mathcal{I})^{1/d}+e_W(\mathcal{J})^{1/d}.$$
			\item[$(iv)$]	Let $\mathcal J=\{J_n\}$ be a bounded below linearly weakly graded family of ideals in $R$ such that $I_n\subset J_n$ and  $\overline{I_n}=\overline{J_n}$ for all $n\gg 0$ then $e_W(\mathcal{I})=e_W(\mathcal{J})$.
		\end{enumerate}
	\end{Theorem}
	\begin{proof} 	Let $s\in \mathbb Z_{>0}$ such that $\mm^{sn}\subset I_n$ for all $n\geq 1$ and $c\in R^o$ such that $cI_mI_n\subset I_{m+n}$ for all $m,n\geq 1$. Note that $\{cI_n\}$ is a graded family of ideals.
		
		By  Artin-Rees Lemma, there exists a positive integer $k$ such that for all $n\geq k$,
		\begin{equation}\label{eq1}cR\cap \mm^n=\mm^{n-k}(cR\cap \mm^k)\subset c\mm^{n-k}.\end{equation}
		
		$(i)$	Let $S=\hat{R}/N(\hat{R})$. Using the technique in \cite{C2014}, we first show that  $\lim\limits_{n\to\infty}\ell_R(R/I_n)/n^d$ exists if $\lim\limits_{n\to\infty}\ell_S(S/I_nS)/n^d$ exists. Note that $\ell_R(R/I_n)=\ell_{\hat{R}}(\hat{R}/I_n\hat{R})$ as $\hat{R}$ is a faithfully flat extension of $R$. Consider the short exact sequence 
		\begin{equation*}\label{nilradical}0\longrightarrow N(\hat{R})/N(\hat{R})\cap I_n\hat{R}\longrightarrow \hat{R}/I_n\hat{R}\longrightarrow \hat{R}/I_n\hat{R}+N(\hat{R})\longrightarrow 0.\end{equation*} Since $\dim(N(\hat{R}))<d$, we have \begin{equation*}\label{zero}\lim\limits_{n\to\infty}\ell_{\hat{R}}(N(\hat{R})/N(\hat{R})\cap I_n\hat{R})/n^d\leq \lim\limits_{n\to\infty}\ell_{\hat{R}}(N(\hat{R})/{\mathfrak m}_{\hat{R}}^{sn}N(\hat{R}))/n^d=0.\end{equation*} Thus the existence of the limit $\lim\limits_{n\to\infty}\ell_S(S/I_nS)/n^d$ implies the existence of the limit $\lim\limits_{n\to\infty}\ell_R(R/I_n)/n^d$. 
		
		By Remark \ref{nzd}, we have $c\in {\hat{R}}^o$ and  hence $c\in S^o$.  We replace $R$ by $S$ and assume the ring $R$ is a complete reduced local Noetherian ring and $c$ is a nonzerodivisor in $R$. 
		
		Let $\Min(R)=\{P_1,\ldots,P_r\}$ and $R_i=R/P_i$ for all $1\leq i\leq r$. 
		Thus by Lemma \ref{reduction}, it is  enough to show that for all $1\leq i\leq r$, $\lim\limits_{n\to\infty}\ell_{R_i}(R_i/I_nR_i)/n^d$ exists.
	
	 Note that $c\in {R_i}^o$ and $\{I_nR_i\}$ is a bounded below linearly  weakly graded family of ideals  for all $1\leq i\leq r$. 
	 Thus by replacing $R$ by $R_i$, we can assume that $R$ is a complete local domain. Note that $\ell_R(R/I_n)=\ell_R(cR/cI_n)$ for all $n\geq 1$.
		
		Consider the two graded families of ideals $\mathcal J=\{J_0=R \mbox{ and }J_n=cR {\mbox{ for all }n\geq 1}\}$ and $\mathcal L=\{L_0=R\mbox{ and }L_n=cI_n{\mbox{ for all }n\geq 1}\}$. Since $\dim R\geq 1$, $c$ is a nonzero element and $I_1$ is an $\mm$-primary ideal, we have  $cI_1$ is a nonzero ideal. 
		
		We show that for $t=s+k+1$ and for all $n\geq 1$, 
		\begin{equation}\label{1eq} cR\cap \mm^{tn}=cI_n\cap \mm^{tn}.\end{equation}
		We already have $cI_n\cap \mm^{tn}\subset cR\cap \mm^{tn}$ for all $n\geq 1$. By equation (\ref{eq1}), we have $$cR\cap \mm^{tn}\subset c\mm^{tn-k}\subset c\mm^{sn}\subset cI_n$$ for all $n\geq 1$. Thus $cR\cap \mm^{tn}\subset cI_n\cap \mm^{tn}$ for  all $n\geq 1$. 
		
		Therefore by \cite[Theorem 6.1]{C2014}, we have the limit $\lim\limits_{n\to\infty}\ell_R(cR/cI_n)/n^d=\lim\limits_{n\to\infty}\ell_R(R/I_n)/n^d$ exists. 
		
		$(ii)$
		Using the same arguments as in part $(i)$, we can replace $R$ by $S=\hat{R}/N(\hat{R})$. Note that $e(I_n)=e(I_nS)$ for all $n\geq 1$. Therefore by Lemma \ref{reduction}, It is enough to prove the result when $R$ is a complete local domain.
		
		Consider the graded families of ideals $\mathcal J(m)=\{J(m)_n\}$ and $\mathcal L(m)=\{L(m)_n\}$ for all $m\geq 1$ in the following way.
		\begin{enumerate}
			\item[$(a)$]  $\mathcal J=\{J_0=R \mbox{ and }J_n=cR {\mbox{ for all }n\geq 1}\}$ and $\mathcal L=\{L_0=R\mbox{ and }L_n=cI_n{\mbox{ for all }n\geq 1}\}$.
			\item[$(b)$] Let $m=1$. Consider $\mathcal J(m)=\mathcal J$ and $\mathcal L(m)=\mathcal L$.
			\item [$(c)$] 	Let $m\geq 2$. Consider $\mathcal J(m)=\{J(m)_n=c^nR\}$ and $\mathcal L(m)=\{L(m)_n=c^nI_m^n\}$.
		\end{enumerate}	
		Then $J(m)_1=J_m$, $L(m)_1=L_m$, $J(m)_n\subset J_{mn}$ and $L(m)_n\subset L_{mn}$ for all $m,n\geq 1$. Since $\dim R\geq 1$, $c$ is a nonzero element and $I_m$ is an $\mm$-primary ideal, we have  $cI_m$ is a nonzero ideal for all $m\geq 1$.
		
		We show that for $t=s+k+1$ and for all $n,m \geq 1$, $$J(m)_n\cap \mm^{tnm}=L(m)_n\cap \mm^{tnm}.$$
		Note that for $m=1$ and for all $n\geq 1$, by equation (\ref{1eq}), we have 		
		$$J(1)_n\cap \mm^{tn}=cR\cap \mm^{tn}=cI_n\cap \mm^{tn}=L(1)_n\cap \mm^{tn}.$$
		Suppose $m\geq 2$. We already have $c^nI_m^n\cap \mm^{tnm}\subset c^nR\cap \mm^{tnm}$ for all $n\geq 1$.
		Note that for all $n\geq 1$, $tnm-\alpha k\geq k$ for all $0\leq\alpha\leq n$.
		Let $c^nx\in c^nR\cap \mm^{tnm}$ for any $n\geq 1$ and for some $x\in R$. Then by  equation (\ref{eq1}), we get $$(c)(c^{n-1}x)\subset cR\cap\mm^{tnm}\subset c\mm^{tnm-k}.$$ As $c$ is a nonzerodivisor, $c^{n-1}x\in \mm^{tnm-k}$. Thus $(c)(c^{n-2}x)\in cR\cap \mm^{tnm-k}\subset c{\mm^{tnm-2k}}$. As $c$ is a nonzerodivisor, $c^{n-2}x \subset {\mm^{tnm-2k}}$. Continuing this process, we get $$x\in \mm^{tnm-nk}\subset\mm^{smn}\subset I_m^n.$$ Hence $c^nx\in c^nI_m^n=L(m)_n$ for all $n\geq 1$. Therefore for all $n\geq 1$ and $m\geq 2$, we have 
		$$J(m)_n\cap \mm^{tnm}=c^nR\cap \mm^{tnm}=c^nI_m^n\cap\mm^{tnm}=L(m)_n\cap \mm^{tnm}.$$ Thus by \cite[Theorem 4.1]{CL2024}, we have 
		\begin{eqnarray*}\lim\limits_{n\to\infty}d!\ell_R(R/I_n)/n^d&=&\lim\limits_{n\to\infty}d!\ell_R(J_n/L_n)/n^d=\lim\limits_{n\to\infty}d!\big(\frac{\lim\limits_{m\to\infty}{\frac{\ell_R(J(n)_m/L(n)_m)}{m^d}}}{n^d}\big)\\&=& \lim\limits_{n\to\infty}d!\big(\frac{\lim\limits_{m\to\infty}\ell_R(R/I_n^m)/m^d}{n^d}\big).
		\end{eqnarray*}
		Therefore by Lemma \ref{reduction}, we get the required result.
		
		$(iii)$ Let $s,u\in \mathbb Z_{>0}$ be
		such that $\mm^{sn}\subset I_n$, $\mm^{un}\subset J_n$ for all $n\geq 1$ and $c,c'\in R^o$ be such that $cI_mI_n\subset I_{m+n}$, $c'J_mJ_n\subset J_{m+n}$ for all $m,n\geq 1$. Then for all $m,n\geq 1$, we have  $cc'I_mJ_mI_nJ_n\subset I_{m+n}J_{m+n}$ and $\mm^{(s+u)n}\subset I_nJ_n$. Hence $\{I_nJ_n\}$ is a bounded below linearly weakly graded family of ideals in $R$.
		
		Let $x_n=e(I_nJ_n)/n^d$, $a_n=e(I_n)/n^d$ and $b_n=e(J_n)/n^d$. Using part $(ii)$, it is enough to show that 
		$$(\lim\limits_{n\to\infty}x_n)^{1/d}\leq (\lim\limits_{n\to\infty}a_n)^{1/d}+(\lim\limits_{n\to\infty}b_n)^{1/d}.$$
		Since $x_n, a_n,b_n$ are non-negative real numbers and $\lim\limits_{n\to\infty}x_n$, $\lim\limits_{n\to\infty}a_n$, $\lim\limits_{n\to\infty}b_n$ exist, we have that 
		$(\lim\limits_{n\to\infty}x_n)^{1/d}=\lim\limits_{n\to\infty}(x_n^{1/d})$, $(\lim\limits_{n\to\infty}a_n)^{1/d}=\lim\limits_{n\to\infty}(a_n^{1/d})$ and $(\lim\limits_{n\to\infty}b_n)^{1/d}=\lim\limits_{n\to\infty}(b_n^{1/d})$.  By \cite[Corollary 17.7.3]{SH}, we have
		
		\begin{eqnarray*}&&{e(I_nJ_n)}^{1/d}\leq {e(I_n)}^{1/d}+{e(J_n)}^{1/d}
			\\&\Rightarrow& {\Big(\frac{e(I_nJ_n)}{n^d}\Big)}^{1/d}\leq {\Big(\frac{e(I_n)}{n^d}\Big)}^{1/d}+{\Big(\frac{e(J_n)}{n^d}\Big)}^{1/d} \\&\Rightarrow& x_n^{1/d}\leq a_n^{1/d}+b_n^{1/d}
			\\&\Rightarrow& \lim\limits_{n\to\infty}(x_n^{1/d})\leq \lim\limits_{n\to\infty}(a_n^{1/d}+b_n^{1/d})= \lim\limits_{n\to\infty}(a_n^{1/d})+\lim\limits_{n\to\infty}(b_n^{1/d})
			\\&\Rightarrow& (\lim\limits_{n\to\infty}x_n)^{1/d}\leq (\lim\limits_{n\to\infty}a_n)^{1/d}+(\lim\limits_{n\to\infty}b_n)^{1/d}.
		\end{eqnarray*} 
		\\$(iv)$		Since $I_n\subset J_n$ and  $\overline{I_n}=\overline{J_n}$ for all $n\gg 0$, we have $e(I_n)=e(J_n)$ for all $n\gg 0$ by \cite[Proposition 11.2.1]{SH}. Thus we get the required result by part $(ii)$.
	\end{proof}	
	\begin{Remark}{\rm
			\begin{enumerate}
				\item Part $(i)$,  part $(ii)$ and  part $(iii)$ of Theorem \ref{existence} give  alternative proofs for Theorems $10.11$,  $10.14$ and $10.16$ in \cite{DM}.
				\item Let $(R,\mm)$ be a Noetherian local ring of prime characteristic $p>0$. A $p$-family of ideals is a sequence of ideals $\{I_q=I_{p^e}\}_{e\in\NN}$ such that $I_q^{[p]}\subset I_{pq}$ (here $I_q^{[p]}$ denotes the Frobenius power of $I_q$) for all $q$ a power of $p$. A family of ideals $\{I_{p^e}\}$ in $R$, is called a bounded below linearly weak $p$-family of ideals if there exist $c\in R^o$ and a positive integer $s$ such that  $cI_q^{[p]}\subset I_{pq}$  and $\mm^{sq}\subset I_q$ for all $q$ a power of $p$. Note that if $\{I_q\}$ is a  bounded below linearly weak $p$-family and $cI_q^{[p]}\subset I_{pq}$ for all $q$ a power of $p$ and for some $c\in R^o$ then $cI_q$ is a $p$-family. Thus using the same lines of proof of Theorem \ref{existence} $(i)$ and \cite[Theorem 1.2]{JH}, we have $\lim\limits_{n\to\infty}d!\ell_R(R/I_q)/q^d$ exists. This gives an alternative proof of \cite[Theorem 10.12]{DM} for $p$-family of ideals.
			\end{enumerate}
		}
	\end{Remark}	
	\begin{Example}\label{countmin}{\rm{\cite[Example 3.2]{CSV}}
		The inequality in part $(iii)$ of Theorem \ref{existence} can be strict.
		
		Let $R=\CC[x,y,z]_{(x,y,z)}/(y^2-x^2(x+1))$ and $\mm$ denote the maximal ideal of $R$. Then $R$ is a two-dimensional excellent local domain. We have the expansion $$x\sqrt{x+1}=a_1x+a_2x^2+a_3x^3+\cdots \mbox{ where } a_{n+1}=\frac{(-1)^{n-	1}(2n-3)!}{2^{2n-2}n!(n-2)!}\mbox{ for }n\geq 2.$$ 
		Consider the ideals  $$G_n=(y-a_1x-a_2x^2-\cdots-a_{n-1}x^{n-1})+\mm^n$$ and $$H_n=(y+a_1x+a_2x^2+\cdots+a_{n-1}x^{n-1})+\mm^n$$ in $R$. Let $\{I_n=G_{\lfloor\frac{n}{2}\rfloor}\}$ and $\{J_n=H_{\lfloor\frac{n}{2}\rfloor}\}$. Then $\{I_n\}$ and $\{J_n\}$ are bounded below linearly weakly graded families of ideals and 
		\begin{eqnarray*}
		\Big(\lim\limits_{n\to\infty}2!\ell_R(R/I_nJ_n)/n^2\Big)^{1/2}&=&1/\sqrt{2}	\\&<&1=1/2+1/2\\&=&\Big(\lim\limits_{n\to\infty}2!\ell_R(R/I_n)/n^2\Big)^{1/2}+\Big(\lim\limits_{n\to\infty}2!\ell_R(R/J_n)/n^2\Big)^{1/2}.
			\end{eqnarray*}	
	}
	\end{Example}	
	\begin{Example}\label{e2}{\rm
			The converse of part $(iv)$ of Theorem \ref{existence} is not true in general. \\Consider the filtrations $\{I_0=R \mbox{ and } I_n=(x^{2n},xy^2,y^{2n})\mbox{ for all }n\geq 1\}$ and $\{J_0=R\mbox{ and }J_n=(x^n,xy,y^n)\mbox{ for all }n\geq 1\}$ in $R=k[x,y]_{(x,y)}$ where $k[x,y]$ is polynomial ring over a field $k$. Then $\overline{I_n:\mm}\neq\overline{J_n:
				\mm}$ for all $n\geq 1$ but $\lim\limits_{n\to\infty}2!\ell_R(R/(I_n:\mm))/n^2=0=\lim\limits_{n\to\infty}2!\ell_R(R/(J_n:\mm))/n^2.$}
		
	\end{Example}	
	\section{the weakly graded family $\{(I_n:K)\}$}
	In this section, we mainly focus on the bounded below linearly weakly graded family of ideals of the form $\{(I_n:K)\}$ where $\{I_n\}$ is a $\mm$-primary graded family of ideals. We show that the limit $\lim\limits_{n\to\infty}d!\ell_R(R/(I_n: K))/n^d$ is bounded above by $e(\mathcal I)$. We discuss some cases where the upper bound is achieved. We also explore a necessary and sufficient condition for the equality of the limits  $\lim\limits_{n\to\infty}d!\ell_R(R/(I_n: K))/n^d$ and $\lim\limits_{n\to\infty}d!\ell_R(R/(J_n: K))/n^d$ for the families $\{(I_n: K)\}$ and $\{(J_n: K)\}$ with $I_n\subset J_n$ for all $n\geq 1$. We provide a necessary and sufficient condition for the equality in Minkowski inequality.
	 
For a weakly graded family of ideals (not necessarily bounded below linearly) of the  form $\{(I_n:K)\}$ where $\mathcal I=\{I_n\}$ is a filtration and satisfies $A(r)$ condition for some $r\in\ZZ_{>0}$, we show that the limit $\lim\limits_{n\to\infty}\ell_R(H_\mm^0(R/(I_{n}:K)))/{n}^{d}$ exists and it is bounded above by $\epsilon(\mathcal I)$ if $K$ is $\mm$-primary. 
\begin{Remark}\label{B2}{\em
		Suppose $\mathcal I=\{I_n\}$ is a Noetherian filtration of ideals in a Noetherian local ring $(R,\mm)$. Then by \cite[Proposition 3, page 159]{Bo}, there exists an integer $m\geq 1$ such that $I_{mn}=I_m^n$ for all $n\geq 1$. Let $K$ be an $\mm$-primary ideal in $R$ and $I_1\neq R$. Then there exists an integer $r\geq 1$ such that $I_{mr}=I_m^r\subset\mm^r\subset K$. Let $e=mr$. Then for all $n\geq 1$, we have $I_{en}=I_{mrn}=I_m^{rn}=I_{mr}^n=I_e^n$.
		
	}	
\end{Remark}	
	\begin{Proposition}\label{tech}
		Let $\mathcal I=\{I_n\}$ be a bounded below linearly weakly graded family of  ideals in a Noetherian local ring $(R,\mm)$ of dimension $d\geq 1$ with $\dim(N(\hat{R}))<d$. Let $K$ be any ideal in $R$. Then the following hold.
		\begin{enumerate} 
			\item[$(i)$]  
			$\lim\limits_{n\to\infty}d!\ell_R(R/(I_n: K))/n^d\leq e_W(\mathcal I)$. 
			\item[$(ii)$] Let $\mathcal I=\{I_n\}$ be a Noetherian filtration of $\mm$-primary ideals in $R$ and $K$ be an $\mm$-primary ideal in $R$. Then $$\lim\limits_{n\to\infty}d!\ell_R(R/(I_n: K))/n^d=e(\mathcal I).$$
			Suppose $R$ is analytically unramified and $I, K$ are $\mm$-primary ideals in $R$. Then $$\lim\limits_{n\to\infty}d!\ell_R(R/(I^n: K))/n^d=\lim\limits_{n\to\infty}d!\ell_R(R/(\overline{I^n}: K))/n^d.$$
			\item[$(iii)$] Suppose $R$ is a local domain. Let $\mathcal I=\{I_n=I(\nu_1)_{\lceil {na_1 }\rceil}\cap\cdots\cap I(\nu_r)_{\lceil {na_r}\rceil}\}$ be a discrete valued filtration with $\mm_{\nu_i}\cap R=\mm$ and $a_i\in\mathbb R_{>0}$ for all $1\leq i\leq r$. Then there exists a positive integer $w$ such that $(I_{wn}: K)\subset I_{w(n-1)}$ for all $n\geq 1$ and $$\lim\limits_{n\to\infty}d!\ell_R(R/(I_n: K))/n^d= e(\mathcal I).$$ 
		\end{enumerate}
	\end{Proposition}	
\begin{proof}
$(i)$ Since $I_n\subset (I_n:K)$ for all $n\geq 1$, we have 
$$ \lim\limits_{n\to\infty}d!\ell_R(R/(I_n:K))/n^d\leq \lim\limits_{n\to\infty}d!\ell_R(R/I_n)/n^d=e_W(\mathcal I).$$
$(ii)$ Let $S=\hat{R}/N(\hat{R})$. Then by the first paragraph of the proof of part $(i)$ of Theorem \ref{existence}, we have $e(\mathcal I)=e(\mathcal IS)$ where $\mathcal IS=\{I_nS\}$ and $\lim\limits_{n\to\infty}d!\ell_R(R/(I_n: K))/n^d=\lim\limits_{n\to\infty}d!\ell_S(S/(I_n: K)S)/n^d$. Hence
\begin{eqnarray*}\lim\limits_{n\to\infty}d!\ell_S(S/(I_nS: KS))/n^d&\leq&\lim\limits_{n\to\infty}d!\ell_S(S/(I_n: K)S)/n^d\\&=&\lim\limits_{n\to\infty}d!\ell_R(R/(I_n: K))/n^d\leq e(\mathcal I).\end{eqnarray*}
Therefore it is enough to show that $e(\mathcal I)=e(\mathcal IS)\leq \lim\limits_{n\to\infty}d!\ell_S(S/(I_nS: KS))/n^d$. Thus we assume $R$ is a reduced complete local ring. Note that $I_n$ are $\mm$-primary ideals and hence $\grade(I_n)\geq 1$ for all $n\geq 1$. Since $\mathcal I$ is a Noetherian filtration, by Remark \ref{B2}, there exists an integer $e\geq 1$ such that $I_e\subset K$ and $I_{en}=I_e^n$ for all $n\geq 1$.
By \cite[Theorem 4.1]{Rat}, for all $n\gg 0$, we have $(I_e^{n}:K)\subset (I_e^{n}:I_e)= I_e^{n-1}.$
 Hence using $(i)$, we get
\begin{eqnarray*}
	 \lim\limits_{n\to\infty}d!\ell_R(R/(I_n: K))/n^d&=&\lim\limits_{n\to\infty}d!\ell_R(R/(I_{en}: K))/{(en)}^d\\&=&\lim\limits_{n\to\infty}d!\ell_R(R/(I_e^n: K))/{(en)}^d\\&\geq & \lim\limits_{n\to\infty}\big[\Big(d!\ell_R(R/I_e^{n-1})/{(e(n-1))}^d\Big)\Big({(e(n-1))}^d/{(en)}^d\Big)\big]\\&=&\lim\limits_{n\to\infty}d!\ell_R(R/I_{e(n-1)})/{(e(n-1))}^d=e(\mathcal I).
\end{eqnarray*}	
Since $R$ is an analytically unramified local ring, by \cite{Rees61}, $\mathcal I=\{\overline{I^n}\}$ is a Noetherian filtration. It is well known that   $e(I)=e(\mathcal I)$. Thus we get the required result.

$(iii)$ If $K=R$ then we get the result. Suppose $K\subseteq \mm$. Let $b_i=\nu_i(K)$ where $\nu_i(K)=\min\{\nu_i(r):r\in K\}$,  $b=b_1\cdots b_r$ and $c_i=b/b_i$ for all $1\leq i\leq r$. Let $y_i\in K$ such that $\nu_i(y_i)=b_i$ for all $1\leq i\leq r$.  

We  have $a_i>0$ for all $1\leq i\leq r$. Now $\lceil{a_i-1}\rceil=a_i-1+g_i$ for some $0\leq g_i<1$. 
Then $a_i-\lceil{a_i-1}\rceil=1-g_i$.
Let $t_i\in\ZZ_{>0}$ be such that $\frac{1}{t_i+1}\leq 1-g_i\leq \frac{1}{t_i}$,  $t=\max\{t_1+1,\cdots,t_r+1\}$ and $w=b(t+1)$.  
	
Suppose $0<a_i<1$.  Then $\lceil{a_i}\rceil=1$, $\lceil{a_i-1}\rceil=0$, $a_i(t+1)-1=(1-g_i)(t+1)-1>0$ and hence  $$\lceil{b(a_i(t+1)-\lceil{a_i }\rceil)}\rceil =\lceil b(a_i(t+1)-1)\rceil\geq 1.$$
	Suppose $a_i\geq 1$. Then $a_i+a_i-\lceil{a_i }\rceil>0$ and hence 
	$$\lceil{b(a_i(t+1)-\lceil{a_i }\rceil)}\rceil \geq \lceil b(2a_i-\lceil{a_i }\rceil)\rceil=\lceil b(a_i+a_i-\lceil{a_i }\rceil)\rceil\geq 1.$$
	Therefore for all $a_i>0$, we have
	$$
	\lceil{a_iw}\rceil-1- \lceil{a_i}\rceil b =\lceil{a_iw-\lceil{a_i}\rceil b}\rceil-1=\lceil{b(a_i(t+1)-\lceil{a_i }\rceil)}\rceil-1\geq 0.$$
	
	Let $z\in (I_{wn}:K)$. Then $zK \subset I_{wn}=\bigcap\limits_{i=1}^rI(\nu_i)_{\lceil{wna_i}\rceil}$. In particular, $zy_i^{c_i\lceil{a_i}\rceil}\in I(\nu_i)_{\lceil{wna_i}\rceil}$ for all $1\leq i\leq r$. Thus for all $1\leq i\leq r$, we have  \begin{eqnarray*}\lceil{a_iw(n-1)}\rceil+\lceil{a_iw}\rceil-1\leq \lceil{a_iwn}\rceil&\leq& \nu_i(y_i^{{c_i}\lceil{a_i}\rceil}z)=\lceil{a_i}\rceil c_i\nu_i(y_i)+\nu_i(z)\\&=&\lceil{a_i}\rceil b+\nu_i(z). \end{eqnarray*} 
	Hence $\nu_i(z)\geq \lceil{a_iw(n-1)}\rceil$ and $z\in I(\nu_i)_{\lceil{a_iw(n-1)}\rceil}$. Therefore $z\in I_{w(n-1)}$. Thus using part $(i)$, we get,
	\begin{eqnarray*}e(\mathcal I)&\geq &\lim\limits_{n\to\infty}d!\ell_R(R/(I_{n}:K))/{n}^d\\&=&\lim\limits_{n\to\infty}d!\ell_R(R/(I_{wn}:K))/{(wn)}^d\\&\geq& \lim\limits_{n\to\infty}\big[\Big(d!\ell_R(R/I_{w(n-1)})/{(w(n-1))}^d \Big)\Big({(w(n-1))}^d/{(wn)}^d\Big)\big]\\&=&
 \lim\limits_{n\to\infty}d!\ell_R(R/I_{w(n-1)})/{(w(n-1))}^d	=e(\mathcal I).\end{eqnarray*}
\end{proof}	
In the next result, we provide  necessary and sufficient conditions for the equality of $\lim\limits_{n\to\infty}d!\ell_R(R/(I_n: K))/n^d$ and $\lim\limits_{n\to\infty}d!\ell_R(R/(J_n: K))/n^d$ and for the equality in Minkowski inequality. 
\begin{Theorem}
	Let $(R,\mm)$ be an analytically irreducible local domain and $K$ be an ideal in $R$.
	\begin{enumerate}
		\item[$(i)$] Let $\mathcal I=\{I_n\}$ be a real bounded filtration of $\mm$-primary ideals and $\mathcal J=\{J_n\}$ be a filtration of $\mm$-primary ideals in $R$ such that $I_n\subset J_n$ for all $n\geq 1$. Then  $\overline{R[\mathcal I]}=\overline{R[\mathcal J]}$ if and only if $$\lim\limits_{n\to\infty}\ell_R(R/(I_n:K))/n^d=\lim\limits_{n\to\infty}\ell_R(R/(J_n:K))/n^d.$$
		\item[$(ii)$] (Minkowski equality) Let $\mathcal I=\{I_n\}$ and $\mathcal J=\{J_n\}$  be two integral bounded filtrations of $\mm$-primary ideals in $R$. Then equality holds in $(iii)$ of Theorem \ref{existence} for the bounded below linearly  weakly graded families of ideals $\{(I_n:K)\}$ and $\{(J_n:K)\}$ if and only if there exist positive integers $a, b$ such that $\overline{\sum\limits_{n\geq 0} I_{an}t^n}=\overline{\sum\limits_{n\geq 0} J_{bn}t^n}$ where the integral closures are in $R[t]$.
	\end{enumerate}
\end{Theorem}	
\begin{proof}
	Let $X=\lim\limits_{n\to\infty}d!\ell_R(R/(I_n:K)(J_n:K))/n^d$, $Y=\lim\limits_{n\to\infty}d!\ell_R(R/(I_n:K)/n^d$ and $Z=\lim\limits_{n\to\infty}d!\ell_R(R/(J_n:K))/n^d$.
	\\$(i)$ Let $\overline{R[\mathcal I]}=R[\mathcal C]$ where $\mathcal C=\{C_n\}$ is a divisorial filtration. By Lemma \ref{integrallyclosed}, we have $R[\mathcal C]=\overline{R[\mathcal C]}$ and by \cite[Theorem 5.1]{CS}, we have $e(\mathcal I)=e(\mathcal C)$.  
	\\Suppose $\overline{R[\mathcal I]}=\overline{R[\mathcal J]}$. Then $I_n\subset J_n\subset C_n$ for all $n\geq 1$ and using  $(i)$ and $(iii)$ of Proposition \ref{tech}, 
	we get, $$e(\mathcal C)=\lim\limits_{n\to\infty}d!\ell_R(R/(C_n:K))/n^d\leq Z\leq Y \leq e(\mathcal I)= e(\mathcal C)$$ which implies $Y=Z$.
	\\Now we prove the converse. Let $Y=Z$. Suppose $\overline{R[\mathcal I]}\neq \overline{R[\mathcal J]}$. 
	By \cite[Theorem 14.4]{C2021}, we have $e(\mathcal I)\neq e(\mathcal J)$. Now $I_n\subset J_n$ for all $n\geq 1$, imply $e(\mathcal J)<e(\mathcal I)=e(\mathcal C)$. Since $I_n\subset C_n$ for all $n\geq 1$, using $(i)$ and $(iii)$ of Proposition \ref{tech}, we get
$$Z\leq e(\mathcal J)< e(\mathcal C)=\lim\limits_{n\to\infty}d!\ell_R(R/(C_n:K))/n^d\leq Y$$ which contradicts the assumption $Y=Z$.
	\\$(ii)$ Due to part $(iii)$  of Theorem \ref{existence}, it is enough to prove that $X^{1/d}\geq Y^{1/d}+Z^{1/d}$. Let $\overline{R[\mathcal I]}=R[\mathcal C]$ and $\overline{R[\mathcal J]}=R[\mathcal T]$ where $\mathcal C=\{C_n\}$ and $\mathcal T=\{T_n\}$ are integral divisorial filtrations. Note that $I_n\subset C_n$ and $J_n\subset T_n$ for all $n\geq 1$. By Lemma \ref{integrallyclosed}, we have $\overline{R[\mathcal I]}=R[\mathcal C]=\overline{R[\mathcal C]}$ and $\overline{R[\mathcal J]}=R[\mathcal T]=\overline{R[\mathcal T]}$. Therefore using \cite[Theorem 5.1]{CS}, part $(i)$ and part $(iii)$ of Proposition \ref{tech}, we have 
	$$e(\mathcal I)=e(\mathcal C)=\lim\limits_{n\to\infty}d!\ell_R(R/(C_n:K))/n^d\leq Y\leq e(\mathcal I)\mbox{ and}$$
	$$e(\mathcal J)=e(\mathcal T)=\lim\limits_{n\to\infty}d!\ell_R(R/(T_n:K))/n^d\leq Z\leq e(\mathcal J).$$ Hence $Y=e(\mathcal C)$ and $Z=e(\mathcal T).$ Therefore by \cite[Theorem 14.5]{C2021}, it is enough to show that $X\geq e(\mathcal H)$ where $\mathcal H=\{C_nT_n\}$.  Note that $$X=\lim\limits_{n\to\infty}d!\ell_R(R/(I_n:K)(J_n:K))/n^d\geq \lim\limits_{n\to\infty}d!\ell_R(R/(C_n:K)(T_n:K))/n^d.$$ Thus it is enough to show that $\lim\limits_{n\to\infty}d!\ell_R(R/(C_n:K)(T_n:K))/n^d\geq e(\mathcal H)$. By part $(iii)$ of  Proposition \ref{tech}, there exist positive integers $w$ and $w'$ such that $(C_{wn}: K)\subset C_{w(n-1)}$ and $(T_{w'n}: K)\subset T_{w'(n-1)}$ for all $n\geq 1$. Let $u=ww'$.  Then for all $n\geq 1$, $$C_{un}\subset(C_{un}:K)=(C_{ww'n}:K)\subset C_{w(w'n-1)}\subset C_{ww'(n-1)}=C_{u(n-1)}\mbox{ and}$$  $$T_{un}\subset(T_{un}:K)=(T_{ww'n}:K)\subset T_{w'(wn-1)}\subset T_{ww'(n-1)}=T_{u(n-1)}.$$ Therefore
	\begin{eqnarray*}e(\mathcal H)&=&\lim\limits_{n\to\infty}d!\ell_R(R/C_{u(n-1)}T_{u(n-1)})/{(u(n-1))}^d\\&\leq& \lim\limits_{n\to\infty}\big[\Big(d!\ell_R(R/(C_{un}:K)(T_{un}:K))/{(un)}^d\Big)\Big({(un)}^d/{(u(n-1))}^d\Big)\big]\\&=& \lim\limits_{n\to\infty}d!\ell_R(R/(C_n:K)(T_n:K))/n^d.\end{eqnarray*} 
	Now $\overline{\sum\limits_{n\geq 0} I_{an}t^n}=\overline{\sum\limits_{n\geq 0} J_{bn}t^n}$ implies ${\sum\limits_{n\geq 0} C_{an}t^n}={\sum\limits_{n\geq 0} T_{bn}t^n}$ and hence by \cite[Theorem 14.5]{C2021} and Minkowski inequality, we have $Y^{1/d}+Z^{1/d}=e(\mathcal C)^{1/d}+e(\mathcal T)^{1/d}=e(\mathcal H)^{1/d}\leq X^{1/d}\leq Y^{1/d}+Z^{1/d}$. For the converse, note that $e(\mathcal I)^{1/d}+e(\mathcal J)^{1/d}=e(\mathcal C)^{1/d}+e(\mathcal T)^{1/d}=Y^{1/d}+Z^{1/d}=X^{1/d}\leq e(\mathcal {IJ})^{1/d}\leq e(\mathcal I)^{1/d}+e(\mathcal J)^{1/d}$ where $\mathcal {IJ}=\{I_nJ_n\}$. Hence the result follows from \cite[Theorem 14.5]{C2021}.
\end{proof}	
Next we explore the asymptotic behaviour of the length function $\ell_R(H_\mm^0(R/(I_{n}:K)))/{n}^{d}$ where $\mathcal I=\{I_n\}$ is a filtration which satisfies $A(r)$ condition for some $r\in\ZZ_{>0}$. 
\begin{Theorem}\label{weakep}
	Let $(R,\mm)$ be an analytically unramified local ring of dimension $d\geq 1$ and $K$ be an ideal in $R$. 
	\begin{enumerate} 
		\item[$(i)$] Let 
		$\mathcal I=\{I_n\}$ be a filtration of ideals in $R$ which satisfies $A(r)$ condition for some $r\in\ZZ_{>0}$. Then the limit $$\lim\limits_{n\to\infty}\ell_R(H_\mm^0(R/(I_{n}:K)))/{n}^{d}$$ exists. 
		\\Suppose $K$ is an $\mm$-primary ideal in $R$. Then $(I_{n}:K)^{\sat}=I_n^{\sat}$ for all $n\geq 1$ and $$\lim\limits_{n\to\infty}d!\ell_R(H_\mm^0(R/(I_{n}:K)))/{n}^{d}\leq \epsilon(\mathcal I).$$ 
		\item[$(ii)$] Let $\mathcal I=\{I_n\}$ be a Noetherian filtration of ideals in $R$ with  $\grade(I_1)\geq 1$ and $K$ be an $\mm$-primary ideal in $R$. Then $$\lim\limits_{n\to\infty}d!\ell_R(H_\mm^0(R/({I_{n}}:K)))/{n}^{d}= \epsilon(\mathcal I).$$ 
		In particular, if $I$ is an ideal in $R$ with $\grade(I)\geq 1$ and $K$ is an $\mm$-primary ideal in $R$ then $$\lim\limits_{n\to\infty}d!\ell_R(H_\mm^0(R/(\overline{I^{n}}:K)))/{n}^{d}=\lim\limits_{n\to\infty}d!\ell_R(H_\mm^0(R/({I^{n}}:K)))/{n}^{d}.$$
	\end{enumerate}
\end{Theorem}	
\begin{proof}
	$(i)$ Let $c\in R^o$ be such that $c(I_m:K)(I_n:K)\subset (I_{m+n}:K)$ for all $m,n\geq 1$. Then $c(I_m:K)^{\sat}(I_n:K)^{\sat}\subset (I_{m+n}:K)^{\sat}$ for all $m,n\geq 1$. Hence $\{(I_n:K)^{\sat}\}$ is a weakly graded family of ideals in $R$. Thus $\{c(I_n:K)\}$ and $\{c(I_n:K)^{\sat}\}$ are filtrations of ideals in $R$. 
	
	Since $\hat{R}$ is a faithfully flat extension of $R$, we can replace $R$, $I_n$ and $K$ by $\hat{R}$, $I_n\hat{R}$ and $K\hat{R}$ respectively. Note that $\hat{R}$ is reduced. Hence by Remark \ref{nzd}, $c$ is a nonzerodivisor. Now $\{I_n\}$  satisfies $A(r)$ condition for some $r\in\ZZ_{>0}$. Hence for all $n\geq 1$, we have \begin{equation}I\label{ep1}_n^{\sat}\cap \mm^{rn}=I_n\cap \mm^{rn}.\end{equation} By  Artin-Rees Lemma, there exists a positive integer $k$ such that for all $n\geq k$,
	$cR\cap \mm^n=\mm^{n-k}(cR\cap \mm^k)\subset c\mm^{n-k}$. Let $l=r+2k$. We show that for all $n\geq 1$,
	$$c(I_n:K)^{\sat}\cap\mm^{ln}=c(I_n:K)\cap\mm^{ln}.$$
We already have $c(I_n:K)\cap\mm^{ln}\subset c(I_n:K)^{\sat}\cap\mm^{ln}$. Let $a\in c(I_n:K)^{\sat}\cap\mm^{ln}$. Then $a=cx$ for some $x\in (I_n:K)^{\sat}$. Therefore $xK\mm^q\subset I_n$ for some $q\in\ZZ_{>0}$ and $xK\subset I_n^{\sat}$. Note that $aK=cxK\subset \mm^{ln}\cap cR\subset c\mm^{ln-k}\subset c\mm^{rn}$. Since $c$ is a nonzerodivisor, we have $xK\subset \mm^{rn}$. Thus by equation (\ref{ep1}), we have $xK\subset I_n^{\sat}\cap\mm^{rn}=I_n\cap\mm^{rn}$. Hence $x\in  (I_n:K)$ and  $a=cx\in c(I_n:K)\cap \mm^{ln}$.
	
	Therefore by \cite[Theorem 6.1]{C2014}, we have the existence of the limit
	\begin{eqnarray*}\lim\limits_{n\to\infty}\ell_R(c(I_n:K)^{\sat}/c(I_n:K))/n^d&=&\lim\limits_{n\to\infty}\ell_R((I_n:K)^{\sat}/(I_n:K))/n^d\\&=&\lim\limits_{n\to\infty}\ell_R(H_\mm^0(R/(I_{n}:K)))/{n}^{d}.\end{eqnarray*}
	Suppose $K$ is an $\mm$-primary ideal. Then there exists $t\in \ZZ_{>0}$ such that $\mm^t\subset K\subset \mm$. Hence $I_n^{\sat}=(I_n:\mm)^{\sat}\subset (I_n:K)^{\sat}\subset (I_n:\mm^t)^{\sat}=I_n^{\sat}$ for all $n\geq 1$. Thus by \cite[Theorem 1.2]{CS2024}, we have
$$\displaystyle\lim\limits_{n\to\infty}d!\ell_R(H_\mm^0(R/(I_{n}:K)))/{n}^{d}\leq \lim\limits_{n\to\infty}d!\ell_R(H_\mm^0(R/I_{n}))/{n}^{d}=\epsilon(\mathcal I).$$
	$(ii)$ Let $\mathcal I=\{I_n\}$ be a Noetherian filtration of ideals in $R$ with $\grade(I_1)\geq 1$. If $I_n=R$ for all $n\geq 1$, then we get the required result. Suppose $I_1 \neq R$. By \cite[Proposition 2.4]{CM} and \cite[Theorem 3.4]{Sw}, there exists $r\in\ZZ_{>0}$ such that $\mathcal I$ satisfies $A(r)$ condition. Hence using  \cite[Theorem 6.1]{C2014}, we have $\epsilon(\mathcal I)$ exists and by $(i)$, $\displaystyle\lim\limits_{n\to\infty}d!\ell_R(H_\mm^0(R/(I_{n}:K)))/{n}^{d}$ exists. By Remark \ref{B2}, there exists an integer $e\geq 1$ such that $I_e\subset K$ and $I_{en}=I_e^n$ for all $n\geq 1$. Therefore using $(i)$, we have 
	\begin{eqnarray}\label{noetherian}\lim\limits_{n\to\infty}d!\ell_R(H_\mm^0(R/({I_{n}}:K)))/{n}^{d}\nonumber&=&\lim\limits_{n\to\infty}d!\ell_R(H_\mm^0(R/({I_{en}}:K)))/{(en)}^{d}\nonumber\\&=&\lim\limits_{n\to\infty}d!\ell_R(H_\mm^0(R/({I_e^{n}}:K)))/{(en)}^{d}.\end{eqnarray}
	Since $\grade(I_1)\geq 1$, we have $\grade(I_e)\geq 1$.
	Consider the ideals $L=\bigoplus\limits_{n\in\mathbb N}I_e^{n+1}$ and $T=K R[I_e]$  in  the  Rees algebra $R[I_e]=\bigoplus\limits_{n\in\NN}I_e^n$. Then  $(J:_{R[I_e]} T)$ is a finitely generated $R[I_e]$-module and hence $M=(L:_{R[I_e]} T)/L$ is a finitely generated $R[I_e]$-module. Note that the $n$th-graded component of $M$ is
	$\displaystyle\frac{(I_e^{n+1}:K)\cap I_e^{n}}{I_e^{n+1}}.$  
	
	By \cite[Theorem 4.1]{Rat}, for all $n\gg 0$, we have $(I_e^{n+1}:K)\subset (I_e^{n+1}:I_e)= I_e^{n}.$ Therefore for all $n\gg 0$, the $n$th-graded component of $M$ is
	$(I_e^{n+1}:K)/I_e^{n+1}.$ Since $K$ is $\mm$-primary, there exists a positive integer $t$ such that $M$ is a finitely generated $R[I_e]/\mm^tR[I_e]$-module. Hence for all $n\gg 0$, $\displaystyle\ell_R((I_e^{n+1}:K)/I_e^{n+1})$ is a polynomial of degree less than or equal to $\ell(I_e)-1\leq d-1$ where $\ell(I_e)$ is the analytic spread of $I_e$. Thus $\displaystyle\lim\limits_{n\to\infty}d!\ell_R(({I_e^{n}}:K)/I_e^{n})/{n}^{d}=0$.
	Therefore by equation (\ref{noetherian}) and part $(i)$, we get
	 \begin{eqnarray}\label{epsilonone}
		\lim\limits_{n\to\infty}d!\ell_R(H_\mm^0(R/({I_{n}}:K)))/{n}^{d}\nonumber&=&\lim\limits_{n\to\infty}d!\ell_R(H_\mm^0(R/({I_e^{n}}:K))/{(en)}^{d}\nonumber\\&=&\lim\limits_{n\to\infty}d!\ell_R({(I_e^n)}^{\sat}/({I_e^{n}}:K))/{(en)}^{d}\nonumber\\&=&\lim\limits_{n\to\infty}d!\ell_R({(I_e^n)}^{\sat}/{I_e^{n}})/{(en)}^{d}-\lim\limits_{n\to\infty}d!\ell_R(({I_e^{n}}:K)/I_e^{n})/{(en)}^{d}\nonumber\\&=&\lim\limits_{n\to\infty}d!\ell_R({(I_e^n)}^{\sat}/{I_e^{n}})/{(en)}^{d}\nonumber\\&=&\lim\limits_{n\to\infty}d!\ell_R({(I_n)}^{\sat}/{I_{n}})/{n}^{d}\nonumber\\&=&\lim\limits_{n\to\infty}d!\ell_R(H_\mm^0(R/{I_{n}}))/{n}^{d}=\epsilon(\mathcal I)\nonumber.
	\end{eqnarray}	 
	
	Since $R$ is an analytically unramified local ring, by \cite{Rees61}, $\mathcal I=\{\overline{I^n}\}$ is a Noetherian filtration. By \cite[Corollary 6.3]{C2014}, $\epsilon(I)=\lim\limits_{n\to\infty}d!\ell_R({(I^n)}^{\sat}/{I^{n}})/{n}^{d}$ exists and by \cite[Proposition 2.1]{JMV},  $\epsilon(\mathcal I)= \lim\limits_{n\to\infty}d!\ell_R(\overline{I^n}^{\sat}/\overline{I^{n}})/{n}^{d}$ exists and $\epsilon(\mathcal I)=\epsilon(I)$. Thus we get the required result.
\end{proof}
\begin{Remark}{\rm
		By \cite[Lemma 4.2]{CQT}, we have 
		$\lim\limits_{n\to\infty}d!\ell_R(({I_e^{n}}:\mm)/I_e^{n})/{n}^{d}=0$. Hence by replacing $K$ by $\mm$ in part $(ii)$ of Theorem \ref{weakep}, for any Noetherian filtration $\mathcal I=\{I_n\}$ (without the assumption that $\grade(I_1)\geq 1$) in an analytically unramified local ring $(R,\mm)$, we get		
		$\lim\limits_{n\to\infty}d!\ell_R(H_\mm^0(R/({I_{n}}:\mm)))/{n}^{d}= \epsilon(\mathcal I)$.
}\end{Remark}
\subsection*{Acknowledgements.}
The author would like to thank Steven Dale Cutkosky for his valuable comments.

\end{document}